\documentclass
{amsart}
\usepackage[dvipdfmx]{graphicx}
\usepackage{tikz}
\usepackage{amssymb,amsmath,amsthm,amscd,amsfonts}

\newtheorem{theorem}{Theorem}[section]
\newtheorem{proposition}[theorem]{Proposition}

\newtheorem{lemma}[theorem]{Lemma}

\theoremstyle{definition}
\newtheorem{definition}[theorem]{Definition}
\newtheorem{remark}[theorem]{Remark}

\begin{document}

%%%%%%%%%%%%%%%%%%%%%%%%%%%%%%%%%%%%%%%%%%%%%%%%%%%%%%%%%%%%%%%%%%%%%%%%%%%%%%%%%%%%%%%%%%%%%%%%%%%%%%%%
%%%%%%%%%%%%%%%%%%%%%%%%%%%%%%%%%%%%%%%%%%%%%%%%%%%%%%%%%%%%%%%%%%%%%%%%%%%%%%%%%%%%%%%%%%%%%%%%%%%%%%%%

\title[]
{Examples of Ricci-mean curvature flows}

%%%%%%%%%%%%%%%%%%%%%%%%%%%%%%%%%%%%%%%%%%%%%%%%%%%%%%%%%%%%%%%%%%%%%%%%%%%%%%%%%%%%%%%%%%%%%%%%%%%%%%%%
%%%%%%%%%%%%%%%%%%%%%%%%%%%%%%%%%%%%%%%%%%%%%%%%%%%%%%%%%%%%%%%%%%%%%%%%%%%%%%%%%%%%%%%%%%%%%%%%%%%%%%%%

\author{Hikaru Yamamoto}
\address{Department of Mathematics, Faculty of Science, Tokyo University of Science, 1-3 Kagurazaka, Shinjuku-ku, Tokyo 162-8601, Japan}
\email{hyamamoto@rs.tus.ac.jp}

%%%%%%%%%%%%%%%%%%%%%%%%%%%%%%%%%%%%%%%%%%%%%%%%%%%%%%%%%%%%%%%%%%%%%%%%%%%%%%%%%%%%%%%%%%%%%%%%%%%%%%%%
%%%%%%%%%%%%%%%%%%%%%%%%%%%%%%%%%%%%%%%%%%%%%%%%%%%%%%%%%%%%%%%%%%%%%%%%%%%%%%%%%%%%%%%%%%%%%%%%%%%%%%%%

%\date{\empty}

%%%%%%%%%%%%%%%%%%%%%%%%%%%%%%%%%%%%%%%%%%%%%%%%%%%%%%%%%%%%%%%%%%%%%%%%%%%%%%%%%%%%%%%%%%%%%%%%%%%%%%%%
%%%%%%%%%%%%%%%%%%%%%%%%%%%%%%%%%%%%%%%%%%%%%%%%%%%%%%%%%%%%%%%%%%%%%%%%%%%%%%%%%%%%%%%%%%%%%%%%%%%%%%%%

\begin{abstract}
Let $\pi:\mathbb{P}(\mathcal{O}(0)\oplus \mathcal{O}(k))\to \mathbb{P}^{n-1}$ be a projective bundle over $\mathbb{P}^{n-1}$ with $1\leq k \leq n-1$. 
In this paper, we show that lens space $L(k\, ;1)(r)$ with radius $r$ embedded in $\mathbb{P}(\mathcal{O}(0)\oplus \mathcal{O}(k))$ is a self-similar solution, 
where $\mathbb{P}(\mathcal{O}(0)\oplus \mathcal{O}(k))$ is endowed with the $U(n)$-invariant gradient shrinking Ricci soliton structure. 
We also prove that there exists a pair of critical radii $r_{1}<r_{2}$ which satisfies the following. 
The lens space $L(k\, ;1)(r)$ is a self-shrinker if $r<r_{2}$ and self-expander if $r_{2}<r$, and 
the Ricci-mean curvature flow emanating from $L(k\, ;1)(r)$ collapses to the zero section of $\pi$ if $r<r_{1}$ and to the $\infty$-section of $\pi$ if $r_{1}<r$. 
This gives explicit examples of Ricci-mean curvature flows. 
\end{abstract} 

%%%%%%%%%%%%%%%%%%%%%%%%%%%%%%%%%%%%%%%%%%%%%%%%%%%%%%%%%%%%%%%%%%%%%%%%%%%%%%%%%%%%%%%%%%%%%%%%%%%%%%%%
%%%%%%%%%%%%%%%%%%%%%%%%%%%%%%%%%%%%%%%%%%%%%%%%%%%%%%%%%%%%%%%%%%%%%%%%%%%%%%%%%%%%%%%%%%%%%%%%%%%%%%%%


%%%%%%%%%%%%%%%%%%%%%%%%%%%%%%%%%%%%%%%%%%%%%%%%%%%%%%%%%%%%%%%%%%%%%%%%%%%%%%%%%%%%%%%%%%%%%%%%%%%%%%%%
%%%%%%%%%%%%%%%%%%%%%%%%%%%%%%%%%%%%%%%%%%%%%%%%%%%%%%%%%%%%%%%%%%%%%%%%%%%%%%%%%%%%%%%%%%%%%%%%%%%%%%%%

%53C42 Immersions (minimal, prescribed curvature, tight, etc.)
%53C44 Geometric evolution equations (mean curvature flow)

%%%%%%%%%%%%%%%%%%%%%%%%%%%%%%%%%%%%%%%%%%%%%%%%%%%%%%%%%%%%%%%%%%%%%%%%%%%%%%%%%%%%%%%%%%%%%%%%%%%%%%%%
%%%%%%%%%%%%%%%%%%%%%%%%%%%%%%%%%%%%%%%%%%%%%%%%%%%%%%%%%%%%%%%%%%%%%%%%%%%%%%%%%%%%%%%%%%%%%%%%%%%%%%%%

%\thanks{This work was supported by...}

%%%%%%%%%%%%%%%%%%%%%%%%%%%%%%%%%%%%%%%%%%%%%%%%%%%%%%%%%%%%%%%%%%%%%%%%%%%%%%%%%%%%%%%%%%%%%%%%%%%%%%%%
%%%%%%%%%%%%%%%%%%%%%%%%%%%%%%%%%%%%%%%%%%%%%%%%%%%%%%%%%%%%%%%%%%%%%%%%%%%%%%%%%%%%%%%%%%%%%%%%%%%%%%%%

%%%%%%%%%%%%%%%%%%%%%%%%%%%%%%%%%%%%%%%%%%%%%%%%%%%%%%%%%%%%%%%%%%%%%%%%%%%%%%%%%%%%%%%%%%%%%%%%%%%%%%%%
%%%%%%%%%%%%%%%%%%%%%%%%%%%%%%%%%%%%%%%%%%%%%%%%%%%%%%%%%%%%%%%%%%%%%%%%%%%%%%%%%%%%%%%%%%%%%%%%%%%%%%%%

\maketitle

%%%%%%%%%%%%%%%%%%%%%%%%%%%%%%%%%%%%%%%%%%%%%%%%%%%%%%%%%%%%%%%%%%%%%%%%%%%%%%%%%%%%%%%%%%%%%%%%%%%%%%%%
%%%%%%%%%%%%%%%%%%%%%%%%%%%%%%%%%%%%%%%%%%%%%%%%%%%%%%%%%%%%%%%%%%%%%%%%%%%%%%%%%%%%%%%%%%%%%%%%%%%%%%%%

%\tableofcontents

%%%%%%%%%%%%%%%%%%%%%%%%%%%%%%%%%%%%%%%%%%%%%%%%%%%%%%%%%%%%%%%%%%%%%%%%%%%%%%%%%%%%%%%%%%%%%%%%%%%%%%%%
%%%%%%%%%%%%%%%%%%%%%%%%%%%%%%%%%%%%%%%%%%%%%%%%%%%%%%%%%%%%%%%%%%%%%%%%%%%%%%%%%%%%%%%%%%%%%%%%%%%%%%%%
\section{Background}\label{BCG}
Let $M$ and $N$ be manifolds, $\mathcal{G}=\{\,g_{t}\mid t\in [0,T)\,\}$ be a smooth 1-parameter family of Riemannian metrics on $N$ and 
$\mathcal{F}=\{\,F_{t}:M\to N\mid t\in [0,T')\,\}$ be a smooth 1-parameter family of immersions with $T'\leq T$. 
\begin{definition}
The pair $(\mathcal{G},\mathcal{F})$ is called a {\it Ricci-mean curvature flow} if it satisfies 
\begin{align}\label{RMCFeq}
\begin{aligned}
\frac{\partial g_{t}}{\partial t}=&-2\mathop{\mathrm{Ric}}(g_{t})\\
\frac{\partial F_{t}}{\partial t}=&H_{g_{t}}(F_{t}), 
\end{aligned}
\end{align}
where $H_{g_{t}}(F_{t})$ denotes the mean curvature vector field of $F_{t}$ calculated with the ambient metric $g_{t}$ at each time $t$. 
\end{definition}
The first equation of (\ref{RMCFeq}) is just the Ricci flow equation on $N$ and it does not depend on existence of $\mathcal{F}$. 
The second equation of (\ref{RMCFeq}) is the mean curvature flow equation though it is affected by the evolution of ambient metics $g_{t}$. 
This is a coupled flow of the Ricci flow and the mean curvature flow. 

The Ricci-mean curvature flow equation has been already appeared in some contexts. 
For example, Smoczyk \cite{Smoczyk2}, Han-Li \cite{HanLi} and Lotay-Pacini \cite{LotayPacini} 
consider the Lagrangian mean curvature flow coupled with the K\"ahler-Ricci flow, 
and generalize several results which hold in Calabi-Yau manifolds to this moving ambient setting. 
Other contexts appear in, for example, Lott \cite{Lott} and Magni-Mantegazza-Tsatis \cite{MagniMantegazzaTsatis}. 
There is a monotonicity formula for a mean curvature flow in a Euclidean space introduced by Huisken \cite{Huisken}. 
They generalized it to a Ricci-mean curvature flow coupled with a Ricci flow constructed by a gradient shrinking Ricci soliton. 

A {\it gradient shrinking Ricci soliton} is a pair of a Riemannian manifold $(N,g)$ and a function $f$ on $N$, called a potential function, which satisfies 
\[\mathop{\mathrm{Ric}}(g)+\mathop{\mathrm{Hess}}f=g. \]
From a given gradient shrinking Ricci soliton $(N,g,f)$ and an arbitrary fixed time $T\in (0,\infty)$ the solution of Ricci flow which survives on $[0,T)$ is constructed by $g_{t}=2(T-t)\Phi_{t}^{*}g$, 
where $\Phi_{t}$ is the 1-parameter family of diffeomorphisms of $N$ generated by $\frac{1}{2(T-t)}\nabla f$ with $\Phi_{0}=\mathrm{id}$. 

Motivated by their works, the author generalized the rescaling procedure due to Huisken to a Ricci-mean mean curvature flow in \cite{Yamamoto2}. 
It states that if a Ricci-mean curvature flow coupled with a Ricci flow $g_{t}=2(T-t)\Phi_{t}^{*}g$ develops singularities of type I at the same time $T$, its rescaling limit is a self-shrinker. 
The definition of self-shrinkers is the follwong. 
\begin{definition}\label{defofselsim}
An immersion $F: M\to N$ to a gradient shrinking Ricci soliton $(N,g,f)$ is called a {\it self-similar solution} with coefficient $\lambda\in\mathbb{R}$ if 
it satisfies 
\begin{align}\label{self3}
H_{g}(F)=\lambda {\nabla f}^{\bot}. 
\end{align}
If $\lambda<0$ or $\lambda>0$, it is called a {\it self-shrinker} or {\it self-expander}, respectively. 
\end{definition}
If $\lambda=0$, it is a minimal immersion. 
Moreover, it also holds that a self-similar solution with coefficient $\lambda$ is a minimal immersion in $N$ 
with respect to a conformally rescaled metric $e^{2\lambda f/m}g$, where $m=\dim M$. 
Hence, self-similar solutions can be considered as a kind of generalization of minimal submanifolds. 
When $(N,g,f)$ is the Gaussian soliton, that is, the Euclidean space with the standard metric and potential function $f=|x|^2/2$, 
the equation (\ref{self3}) is written as $H(F)=\lambda x^{\bot}$. 
Thus, Definition \ref{defofselsim} coincides with the ordinary notion of self-similar solutions in the Euclidean space. 
The original form of Definition \ref{defofselsim} for general Ricci solitons has been appeared in \cite{Lott}. 

It is well-known that a mean curvature flow in a fixed Riemannian manifold is a backward $L^2$-gradient flow of the volume functional, 
and a Ricci flow can be regarded as a gradient flow of Perelman's $\mathcal{W}$-entropy functional. 
As mentioned above, Ricci-mean curvature flows have some common properties with mean curvature flows in Euclidean spaces or more generally Ricci flat spaces. 
However, it becomes unclear that a Ricci-mean curvature flow can be considered as a backward $L^2$-gradient flow of some functional. 
To the authors knowledge, a problem to fined an appropriate functional such that its gradient flow is a Ricci-mean curvature flow is still open. 

If $M$ and $N$ are compact, for any Riemannian metric $g$ on $N$ and immersion $F:M\to N$, 
the short-time existence and uniqueness of the Ricci-mean curvature flow equation (\ref{RMCFeq}) with initial condition $(g,F)$ are assured. 
Actually, first, construct a unique short-time solution of Ricci flow on $N$ with initial metric $g$, and next, solve the second equation of (\ref{RMCFeq}) for short-time with initial immersion $F$. 
Hence, there are infinitely many examples of Ricci-mean curvature flows. 

However, explicit examples of Ricci-mean curvature flows are not known, so far. 
In this paper, we consider a lens space $L(k\,;1)(r)$ embedded in $\mathbb{P}(\mathcal{O}(0)\oplus \mathcal{O}(k))$ endowed with a gradient shrinking Ricci soliton structure, 
and investigate how it moves along the Ricci-mean curvature flow. 
The analysis is done by reducing PDE (\ref{RMCFeq}) to ODE (\ref{ODEradi2}). 
This example shows how the evolution of the ambient metrics affects the motion of submanifolds. 
To the authors knowledge, this gives a first non-trivial explicit example of Ricci-mean curvature flow, 
and the author hopes that this example inspires further research of Ricci-mean curvature flows. \\

%%%%%%%%%%%%%%%%%%%%%%%%%%%%%%%%%%%%%%%%%%%%%%%%%%%%%%%%%%%%%%%%%%%%%%%%%%%%%%%%%%%%%%%%%%%%%%%%%%%%%%%%
\noindent
{\bf Organization of this paper.} 
Section \ref{BCG} is a background of a study of Ricci-mean curvature flows, 
and definitions of Ricci-mean curvature flows and self-similar solutions are given in it. 
In Section \ref{sectmain}, we briefly summarize main results of this paper. 
In Section \ref{caoconst}, we review a construction of a gradient shrinking K\"ahler Ricci soliton structure on $\mathbb{P}(\mathcal{O}(0)\oplus \mathcal{O}(k))$. 
In Section \ref{caoconst}, we see that lens spaces $L(k\,;1)(r)$ with radius $r$ are naturally embedded in $\mathbb{P}(\mathcal{O}(0)\oplus \mathcal{O}(k))$ 
and these are self-similar solutions. 
In Section \ref{motion}, we investigate the motion of the Ricci-mean curvature flow emanating from $L(k\,;1)(r)$. 
In Section \ref{motionh}, we compare it to the ordinary mean curvature flow emanating from $L(k\,;1)(r)$. 
%%%%%%%%%%%%%%%%%%%%%%%%%%%%%%%%%%%%%%%%%%%%%%%%%%%%%%%%%%%%%%%%%%%%%%%%%%%%%%%%%%%%%%%%%%%%%%%%%%%%%%%%
%%%%%%%%%%%%%%%%%%%%%%%%%%%%%%%%%%%%%%%%%%%%%%%%%%%%%%%%%%%%%%%%%%%%%%%%%%%%%%%%%%%%%%%%%%%%%%%%%%%%%%%%
%%%%%%%%%%%%%%%%%%%%%%%%%%%%%%%%%%%%%%%%%%%%%%%%%%%%%%%%%%%%%%%%%%%%%%%%%%%%%%%%%%%%%%%%%%%%%%%%%%%%%%%%
%%%%%%%%%%%%%%%%%%%%%%%%%%%%%%%%%%%%%%%%%%%%%%%%%%%%%%%%%%%%%%%%%%%%%%%%%%%%%%%%%%%%%%%%%%%%%%%%%%%%%%%%
%%%%%%%%%%%%%%%%%%%%%%%%%%%%%%%%%%%%%%%%%%%%%%%%%%%%%%%%%%%%%%%%%%%%%%%%%%%%%%%%%%%%%%%%%%%%%%%%%%%%%%%%
%%%%%%%%%%%%%%%%%%%%%%%%%%%%%%%%%%%%%%%%%%%%%%%%%%%%%%%%%%%%%%%%%%%%%%%%%%%%%%%%%%%%%%%%%%%%%%%%%%%%%%%%
%%%%%%%%%%%%%%%%%%%%%%%%%%%%%%%%%%%%%%%%%%%%%%%%%%%%%%%%%%%%%%%%%%%%%%%%%%%%%%%%%%%%%%%%%%%%%%%%%%%%%%%%
%%%%%%%%%%%%%%%%%%%%%%%%%%%%%%%%%%%%%%%%%%%%%%%%%%%%%%%%%%%%%%%%%%%%%%%%%%%%%%%%%%%%%%%%%%%%%%%%%%%%%%%%
%%%%%%%%%%%%%%%%%%%%%%%%%%%%%%%%%%%%%%%%%%%%%%%%%%%%%%%%%%%%%%%%%%%%%%%%%%%%%%%%%%%%%%%%%%%%%%%%%%%%%%%%
%%%%%%%%%%%%%%%%%%%%%%%%%%%%%%%%%%%%%%%%%%%%%%%%%%%%%%%%%%%%%%%%%%%%%%%%%%%%%%%%%%%%%%%%%%%%%%%%%%%%%%%%
%%%%%%%%%%%%%%%%%%%%%%%%%%%%%%%%%%%%%%%%%%%%%%%%%%%%%%%%%%%%%%%%%%%%%%%%%%%%%%%%%%%%%%%%%%%%%%%%%%%%%%%%
\section{Main Results}\label{sectmain}
We give a summary of results of this paper in the following. 
An ambient space is the $k$-twisted projective-line bundle 
$\mathbb{P}(\mathcal{O}(0)\oplus \mathcal{O}(k))$ over $\mathbb{P}^{n-1}$, where $n\geq 2$ and $1\leq k \leq n-1$, and we denote it by $N_{k}^{n}$. 
It can be shown that $N_{k}^{n}$ contains $(\mathbb{C}^{n}\setminus\{0\})/\mathbb{Z}_{k}$ as an open dense subset, and 
its complement is the disjoint union of $S_{0}$ and $S_{\infty}$, where these denote the image of $0$-section and $\infty$-section respectively. 
The eliminated $\{0\}$ and the point at infinity of $\mathbb{C}^{n}\setminus\{0\}$ are replaced by $S_{0}$ and $S_{\infty}$ respectively. 
See Figure \ref{fig1}. 

It is known, by Cao \cite{Cao} and Koiso \cite{Koiso}, that there exists a unique $U(n)$-invariant gradient shrinking K\"ahler Ricci soliton structure on $N_{k}^{n}$. 
We denote its Riemannian metric and potential function by $g$ and $f$, respectively. 

\begin{figure}[h]
\begin{center}
\includegraphics[bb=150 575 390 710, clip]{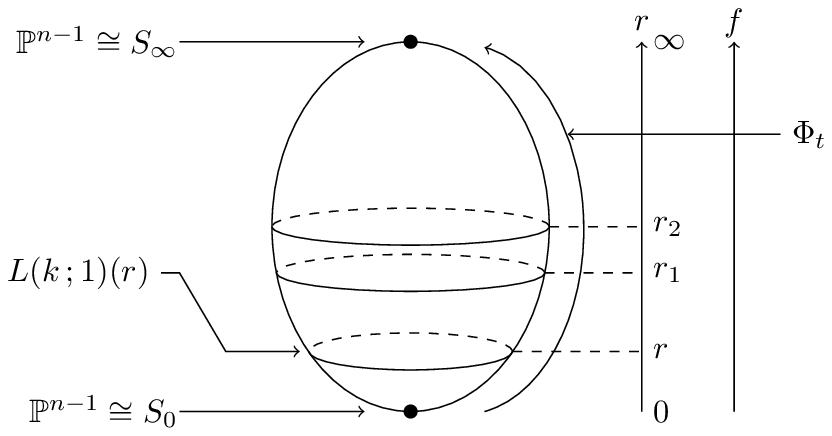}
%\begin{tikzpicture}
%%%%\draw[step=0.5,very thin, gray] (0,0) grid (11.5,5);%help line%%%%
%\draw (5.5,2.5) ellipse (1.5 and 2);%body
%\draw (4,2.5) arc (-180:0: 1.5 and 0.2);%r_{2}-circle(d)
%\draw[dashed] (7,2.5) arc (0:180:1.5 and 0.2);%r_{2}-circle(u)
%\draw[dashed] (7,2.5)--(8,2.5);%r_{2}-line
%\node[anchor=west] at (8,2.5){$r_{2}$};%r_{2}
%\draw (4.05,2) arc (-180:0:1.45 and 0.2);%r_{1}-circle(d)
%\draw[dashed] (6.95,2) arc (0:180:1.45 and 0.2);%r_{1}-circle(d)
%\draw[dashed] (6.95,2)--(8,2);%r_{1}-line
%\node[anchor=west] at (8,2){$r_{1}$};%r_{1}
%\draw (4.4,1.15) arc (-180:0:1.1 and 0.2);%r-circle(d)
%\draw[dashed] (6.6,1.15) arc (0:180:1.1 and 0.2);%r-circle(d)
%\draw[dashed] (6.6,1.15)--(8,1.15);%r-line
%\node[anchor=west] at (8,1.15){$r$};%r-revel
%\draw[->] (2.8,2)--(3,2)--(3.5,1.15)--(4.3,1.15);%L-arrow
%\node[anchor=east] at (2.8,2){$L(k\,;1)(r)$};%L
%\fill (5.5,4.5) circle [radius=0.075];%S_{\infty}point
%\draw[->] (3,4.5)--(5,4.5);%S_{\infty}arrow
%\node[anchor=east] at (3.1,4.5){$\mathbb{P}^{n-1}\cong S_{\infty}$};%S_{\infty}arrow
%\fill (5.5,0.5) circle [radius=0.075];%S_{0}point
%\draw[->] (3,0.5)--(5,0.5);%S_{0}arrow
%\node[anchor=east] at (3.1,0.5){$\mathbb{P}^{n-1}\cong S_{0}$};%S_{0}arrow
%\draw[->] (8,0.5)--(8,4.5);%r-arrow
%\node at (8,4.7){$r$};%r
%\node[anchor=west] at (8,4.5){$\infty$};%r=\infty
%\node[anchor=west] at (8,0.5){$0$};%r=0
%\draw[->] (9,0.5)--(9,4.5);%f-arrow
%\node at (9,4.7){$f$};%f
%\draw[->] (6.3,0.5) arc (-80:80:1.3 and 2);%diffeo
%\draw[->] (9.5,3.5)--(7.2,3.5);%diffeo-arrow
%\node[anchor=west] at (9.5,3.5){$\Phi_{t}$};%Phi
%\end{tikzpicture}
\end{center}
\caption{$N_{k}^{n}$}
\label{fig1}
\end{figure}

For $0<r<\infty$, we consider the $\mathbb{Z}_{k}$ quotient of $S^{2n-1}(r)$, the sphere with radius $r$ in $\mathbb{C}^{n}$, and denote it by $L(k\,;1)(r)$. 
Actually, it is a lens space and embedded in $N_{k}^{n}$. 
We denote its inclusion map by 
\[\iota_{r}:L(k\,;1)(r)\hookrightarrow N_{k}^{n}. \]
Then, the first theorem states that these are self-similar solutions, 
and whether it is a self-shrinker or self-expander is distinguished by whether its radius is smaller or larger than a critical radius $r_{2}$. 
\begin{theorem}\label{sumofmain1}
For each $0<r<\infty$, the inclusion map $\iota_{r}:L(k\,;1)(r)\hookrightarrow N_{k}^{n}$ is a compact self-similar solution. 
Furthermore, there exists a unique radius $r_{2}$ which satisfies the following. 
\begin{itemize}
\item If $r<r_{2}$, $\iota_{r}:L(k\,;1)(r)\hookrightarrow N_{k}^{n}$ is a non-minimal self-shrinker. 
\item If $r=r_{2}$, $\iota_{r}:L(k\,;1)(r)\hookrightarrow N_{k}^{n}$ is a minimal embedding. 
\item If $r>r_{2}$, $\iota_{r}:L(k\,;1)(r)\hookrightarrow N_{k}^{n}$ is a non-minimal self-expander. 
\end{itemize}
\end{theorem}

\begin{remark}
Especially, Theorem \ref{mainthm} claims that there exists a compact self-expander in $N_{k}^{n}$. 
To contrast with the case that the ambient space is a Euclidean space, we remark that there exists no compact self-expander in $\mathbb{R}^{n}$. 
It can be proved by several ways, for instance, see Proposition 5.3 in \cite{CaoLi}. 
In $\mathbb{R}^{n}$, the sphere $S^{n-1}(r)$ is a self-shrinker for every radius $r$. 
However, intuitively, we get a self-expander in $(N_{k}^{n},\omega, f)$ by taking the radius sufficiently large because of bending and compactifying 
the neighborhood of $\{\infty\}$ of $(\mathbb{C}^{n}\setminus\{0\})/\mathbb{Z}_{k}$. 
\end{remark}

Fix a time $0<T<\infty$. 
Then, we will check that the 1-parameter family of diffeomorphisms $\Phi_{t}$ generated by $\frac{1}{2(T-t)}\nabla f$ with $\Phi_{0}=\mathrm{id}$ is given by 
\[\Phi_{t}(z):=\left(\frac{T}{T-t}\right)^{\frac{c}{2}}z\]
for $t\in[0,T)$ on $(\mathbb{C}^{n}\setminus\{0\})/\mathbb{Z}_{k}$. 
Here, $c$ is a positive constant defined in the process to construct the soliton structure, and we skip its explanation here. 
Then, we obtain a Ricci flow 
\[g_{t}:=2(T-t)\Phi_{t}^{*}g\]
which survives on the time interval $[0,T)$. 
We remark that, since $c$ is positive, $\Phi_{t}$ is expanding and $\Phi_{t}(z)$ converges to a point in $S_{\infty}$ as $t\to T$ if $z$ is not contained in $S_{0}$. 
See Figure \ref{fig1}. 

For a fixed radius $r$, we take the solution of Ricci-mean curvature flow $F_{t}:L(k\,;1)(r)\to N_{k}^{n}$ along $g_{t}$ with initial condition $F_{0}=\iota_{r}$. 
We assume that $F_{t}$ exists on $[0,T')$ and $T'(=T'(r))$ is the maximal time of existence of the solution. 
Note that $T'\leq T$ in general. 
Then, the following is a summary of Theorem \ref{mainRM}. 
\begin{theorem}\label{sumofmain2}
There exists a unique radius $r_{1}$ with $r_{1}<r_{2}$ which satisfies the following. 
\begin{itemize}
\item $T'=T$ if and only if $r=r_{1}$. 
\item If $r\leq r_{1}$, $F_{t}:L(k\,;1)(r)\to N_{k}^{n}$ collapses to $S_{0}$ as $t\to T'$. 
\item If $r>r_{1}$, $F_{t}:L(k\,;1)(r)\to N_{k}^{n}$ collapses to $S_{\infty}$ as $t\to T'$. 
\end{itemize}
\end{theorem}

Theorem \ref{mainRM} contains further information about the blow up rate of the solution. 
Actually, we see that the blow up rate is type I in each case. 
The above theorem reveals how a lens space $L(k\,;1)(r)$ moves and what it converges to by a Ricci-mean curvature flow along the Ricci flow $g_{t}$. 
On the other hand, in section \ref{motionh}, we investigate the evolution of $L(k\,;1)(r)$ by the ordinary mean curvature flow in the fixed Riemannian manifold $(N_{k}^{n},g)$. 
Then, we prove that if $r<r_{2}$ ($r>r_{2}$) it collapses to $S_{0}$ ($S_{\infty}$) in finite time and its blow up rate is also type I. 
Of course, if $r=r_{2}$, $L(k\,;1)(r)$ does not move since $L(k\,;1)(r)$ is minimal. 
Thus, the critical radius $r_{1}$ which determine whether a lens space tends to $S_{0}$-side or $S_{\infty}$-side under the Ricci-mean curvature flow is smaller than 
the minimal radius $r_{2}$. See Figure \ref{fig1}. 
Here we summarize the situation on Table \ref{tab1}. 

\begin{table}[h]
\caption{Ricci-mean curvature flow and mean curvature flow}
\label{tab1}
\begin{tabular}{|c||c|c|c|c|c|c|}
\multicolumn{6}{c}{Ricci-mean curvature flow}\\
\hline
 Radius $r$ & $r<r_{1}$ & $r_{1}$ & $r_{1}<r<r_{2}$ & $r_{2}$ & $r_{2}<r$ \\
\hline
Maximal time $T'$ & $T'<T$ & $T'=T$ & \multicolumn{3}{|c|}{$T'<T$} \\
\hline
Collapse to & \multicolumn{2}{|c|}{$S_{0}$} & \multicolumn{3}{|c|}{$S_{\infty}$} \\
\hline
blow up rate & \multicolumn{5}{|c|}{Type I}\\
 \hline
\multicolumn{6}{c}{}\\
\multicolumn{6}{c}{Mean curvature flow}\\
\hline
 Radius $r$ & $r<r_{1}$ & $r_{1}$ & $r_{1}<r<r_{2}$ & $r_{2}$ & $r_{2}<r$ \\
\hline
Maximal time $T'$ & \multicolumn{3}{|c|}{$T'<\infty$}  & $T'=\infty$ & $T'<\infty$ \\
\hline
Collapse to & \multicolumn{3}{|c|}{$S_{0}$} & --- & $S_{\infty}$ \\
\hline
blow up rate & \multicolumn{3}{|c|}{Type I} & --- & Type I \\
\hline
\end{tabular}
\end{table}
%%%%%%%%%%%%%%%%%%%%%%%%%%%%%%%%%%%%%%%%%%%%%%%%%%%%%%%%%%%%%%%%%%%%%%%%%%%%%%%%%%%%%%%%%%%%%%%%%%%%%%%%
%%%%%%%%%%%%%%%%%%%%%%%%%%%%%%%%%%%%%%%%%%%%%%%%%%%%%%%%%%%%%%%%%%%%%%%%%%%%%%%%%%%%%%%%%%%%%%%%%%%%%%%%
%%%%%%%%%%%%%%%%%%%%%%%%%%%%%%%%%%%%%%%%%%%%%%%%%%%%%%%%%%%%%%%%%%%%%%%%%%%%%%%%%%%%%%%%%%%%%%%%%%%%%%%%
%%%%%%%%%%%%%%%%%%%%%%%%%%%%%%%%%%%%%%%%%%%%%%%%%%%%%%%%%%%%%%%%%%%%%%%%%%%%%%%%%%%%%%%%%%%%%%%%%%%%%%%%
%%%%%%%%%%%%%%%%%%%%%%%%%%%%%%%%%%%%%%%%%%%%%%%%%%%%%%%%%%%%%%%%%%%%%%%%%%%%%%%%%%%%%%%%%%%%%%%%%%%%%%%%
%%%%%%%%%%%%%%%%%%%%%%%%%%%%%%%%%%%%%%%%%%%%%%%%%%%%%%%%%%%%%%%%%%%%%%%%%%%%%%%%%%%%%%%%%%%%%%%%%%%%%%%%
%%%%%%%%%%%%%%%%%%%%%%%%%%%%%%%%%%%%%%%%%%%%%%%%%%%%%%%%%%%%%%%%%%%%%%%%%%%%%%%%%%%%%%%%%%%%%%%%%%%%%%%%
%%%%%%%%%%%%%%%%%%%%%%%%%%%%%%%%%%%%%%%%%%%%%%%%%%%%%%%%%%%%%%%%%%%%%%%%%%%%%%%%%%%%%%%%%%%%%%%%%%%%%%%%
%%%%%%%%%%%%%%%%%%%%%%%%%%%%%%%%%%%%%%%%%%%%%%%%%%%%%%%%%%%%%%%%%%%%%%%%%%%%%%%%%%%%%%%%%%%%%%%%%%%%%%%%
%%%%%%%%%%%%%%%%%%%%%%%%%%%%%%%%%%%%%%%%%%%%%%%%%%%%%%%%%%%%%%%%%%%%%%%%%%%%%%%%%%%%%%%%%%%%%%%%%%%%%%%%
%%%%%%%%%%%%%%%%%%%%%%%%%%%%%%%%%%%%%%%%%%%%%%%%%%%%%%%%%%%%%%%%%%%%%%%%%%%%%%%%%%%%%%%%%%%%%%%%%%%%%%%%
\section{Quick review of Cao's construction}\label{caoconst}
The first example of non-trivial compact gradient shrinking Ricci soliton was found by Koiso \cite{Koiso} and independently by Cao \cite{Cao}, and it is actually a K\"ahler Ricci soliton. 
In this section, we quickly review the construction of it following Section 4 in \cite{Cao} and also Section 7.2 in \cite{Chow}. 
Fix integers $n$, $k$ with $n\geq 2$ and $1\leq k \leq n-1$, and consider the $k$-twisted projective-line bundle 
\[\pi:\mathbb{P}(\mathcal{O}(0)\oplus \mathcal{O}(k))\to \mathbb{P}^{n-1}\]
over $\mathbb{P}^{n-1}$, where $\mathcal{O}(k)$ denotes the $k$-th tensor power of the hyperplane bundle $\mathcal{O}(1)$ over $\mathbb{P}^{n-1}$. 
The gradient shrinking K\"ahler Ricci soliton is constructed on $\mathbb{P}(\mathcal{O}(0)\oplus \mathcal{O}(k))$. 
Let $(z^{1}:\cdots:z^{n})$ be the homogeneous coordinates on $\mathbb{P}^{n-1}$, and put $U_{j}:=\{\,z^{j}\neq 0\,\}\subset\mathbb{P}^{n}$ for $j=1,\dots ,n$. 
Then $\{U_{1},\dots,U_{n}\}$ gives an open covering of $\mathbb{P}^{n-1}$, and the transition functions of $\mathcal{O}(k)$ are given by 
\begin{align}\label{trans}
y^{i}={\left(\frac{z^{i}}{z^{j}}\right)}^{k}y^{j}
\end{align}
over $U_{i}\cap U_{j}$, where $y^{i}\in\mathbb{C}$ is the standard coordinate of a fiber of $\mathcal{O}(k)$ over $U_{i}$. 
For a point $(z^{1},\dots,z^{n}) \in \mathbb{C}^{n}\setminus\{0\}$ with $z^{j}\neq 0$, we define 
\[\psi(z^{1},\dots,z^{n}):=((z^{1}:\cdots:z^{n}),(1:(z^{j})^{k}))\in U_{j}\times\mathbb{P}^1 \cong \pi^{-1}(U_{j}). \]
This definition is compatible for $(z^{1},\dots,z^{n}) \in \mathbb{C}^{n}\setminus\{0\}$ with $z^{i}\neq 0$ and $z^{j}\neq 0$ since 
$(1:(z^{j})^{k})\in \mathbb{P}(\mathbb{C}\oplus\mathbb{C})$ in the fiber over $U_{j}$ is identified with 
$(1:(z^{i})^{k})\in \mathbb{P}(\mathbb{C}\oplus\mathbb{C})$ in the fiber in $U_{i}$ by the relation (\ref{trans}). 
Hence we have a smooth map 
 \[\psi:\mathbb{C}^{n}\setminus\{0\}\to \mathbb{P}(\mathcal{O}(0)\oplus \mathcal{O}(k)). \]
It is clear that $\psi(z^{1},\dots,z^{n})=\psi(z'^{1},\dots,z'^{n})$ if and only if 
$z'=e^{2\pi i\frac{\ell}{k}} z$ for some $\ell\in\mathbb{Z}$. 
Hence $\psi$ induces an open dense embedding 
\[\hat{\psi}:(\mathbb{C}^{n}\setminus\{0\})/\mathbb{Z}_{k} \hookrightarrow \mathbb{P}(\mathcal{O}(0)\oplus \mathcal{O}(k)), \]
where the $\mathbb{Z}_{k}$-action on $\mathbb{C}^{n}\setminus\{0\}$ is defined by $([\ell],z)\mapsto e^{2\pi i\frac{\ell}{k}} z$, 
and the complement of the image of $\hat{\psi}$ is $S_{0}\sqcup S_{\infty}$, where $S_{0}$ and $S_{\infty}$ denote the image of $0$-section and $\infty$-section of $\pi$, respectively. 

From now on, we denote $\mathbb{P}(\mathcal{O}(0)\oplus \mathcal{O}(k))$ by $N_{k}^{n}$, and we identify $(\mathbb{C}^{n}\setminus\{0\})/\mathbb{Z}_{k}$ with its image of $\hat{\psi}$. 
Thus, we have an open dense subset
\[(\mathbb{C}^{n}\setminus\{0\})/\mathbb{Z}_{k}\subset N_{k}^{n}. \]
The K\"ahler Ricci soliton structure is constructed on $(\mathbb{C}^{n}\setminus\{0\})/\mathbb{Z}_{k}$, and it actually extends smoothly to $S_{0}$ and $S_{\infty}$. 
Let $u:\mathbb{R}\to\mathbb{R}$ be a smooth function which satisfies 
\begin{align}\label{udash}
u'(s)>0\quad\mathrm{and}\quad u''(s)>0, 
\end{align}
and has the following asymptotic expansions 
\begin{align}\label{uasy}
\begin{aligned}
u(s)=&(n-k)s+a_{1}e^{ks}+a_{2}e^{2ks}+\cdots \quad(s\to-\infty)\\
u(s)=&(n+k)s+b_{1}e^{-ks}+b_{2}e^{-2ks}+\cdots \quad(s\to\infty)\\
\end{aligned}
\end{align}
with $a_{1}>0$ and $b_{1}>0$. 
Define a $U(n)$-invariant smooth function $\Phi:\mathbb{C}^{n}\setminus\{0\}\to\mathbb{R}$ by
\[\Phi(z):=u(s)\quad\mathrm{with}\quad s=\log |z|^2. \] 
Since $\Phi$ is $\mathbb{Z}_{k}$-invariant, it induce a smooth function on $(\mathbb{C}^{n}\setminus\{0\})/\mathbb{Z}_{k}$, and 
we continue to denote it by $\Phi$. 
By the positivity conditions (\ref{udash}), we get a K\"ahler form 
\begin{align}\label{omega}
\omega=\sqrt{-1}\frac{\partial^2 \Phi}{\partial z^{\alpha}\partial \bar{z}^{\beta}}dz^{\alpha}\wedge d\bar{z}^{\beta}
\end{align}
on $(\mathbb{C}^{n}\setminus\{0\})/\mathbb{Z}_{k}$, where $(z^{1},\dots,z^{n})$ is the ($k$-to-one) global holomorphic coordinates. 
By the asymptotic conditions (\ref{uasy}), K\"ahler form $\omega$ extends smoothly to $S_{0}$ and $S_{\infty}$, 
and we get a global K\"ahler structure on $N_{k}^{n}$. 
The Ricci form of $\omega$ is 
\[\mathop{\mathrm{Ric}}(\omega)=-\sqrt{-1}\frac{\partial^2 \log\det(g)}{\partial z^{\alpha}\partial \bar{z}^{\beta}}dz^{\alpha}\wedge d\bar{z}^{\beta}, \]
where $g=(g_{\alpha\bar{\beta}})$ is a matrix given by 
\begin{align}\label{g}
g_{\alpha\bar{\beta}}=\frac{\partial^2 \Phi}{\partial z^{\alpha}\partial \bar{z}^{\beta}}=e^{-s}u'(s)\delta_{\alpha\bar{\beta}}+e^{-2s}\bar{z}^{\alpha}z^{\beta}(u''(s)-u'(s))
\end{align}
for $s=\log|z|^2$. One can easily check that
\begin{align*}
g^{\alpha\bar{\beta}}(z)=&\frac{e^{s}}{u'(s)}\delta^{\alpha\bar{\beta}}+z^{\alpha}\bar{z}^{\beta}\left(\frac{1}{u''(s)}-\frac{1}{u'(s)}\right), \\
\det(g(z))=&e^{-ns}(u'(s))^{n-1}u''(s). 
\end{align*}
Define a real valued smooth function $P:\mathbb{R}\to\mathbb{R}$ by 
\begin{align}\label{f}
\begin{aligned}
P(s):=&\log\left(e^{-ns}(u'(s))^{n-1}u''(s)\right)+u(s)\\
=&-ns+(n-1)\log u'(s)+\log u''(s)+u(s), 
\end{aligned}
\end{align}
and a $U(n)$-invariant real valued smooth function $f:\mathbb{C}^{n}\setminus\{0\}\to\mathbb{R}$ by 
\begin{align}\label{F}
f(z):=P(s)=\log\det(g(z))+\Phi(z) \quad\mathrm{with}\quad s=\log |z|^2.
\end{align}
Since $f$ is $\mathbb{Z}_{k}$-invariant, it induce a smooth function on $(\mathbb{C}^{n}\setminus\{0\})/\mathbb{Z}_{k}$, and 
we continue to denote it by $f$. 
Then, we have 
\[\mathop{\mathrm{Ric}}(\omega)+\sqrt{-1}\partial\bar{\partial}f=\omega. \]
This equation is just the (1,1)-part of the gradient shrinking Ricci soliton equation 
\begin{align}\label{ksoliton}
\mathop{\mathrm{Ric}}+\mathop{\mathrm{Hess}}f=g, 
\end{align}
where $g$ is the associated Riemannian metric of $\omega$ and $\mathop{\mathrm{Ric}}$ is the Ricci 2-tensor of $g$. 
Thus, the property that $f$ satisfies (\ref{ksoliton}) is equivalent to that $\nabla f$ is a holomorphic vector field. 
The coefficient of $\partial/\partial z^{\alpha}$ of $\nabla f$ is given by 
\begin{align*}
g^{\alpha\bar{\beta}}\frac{\partial f}{\partial{\bar{z}}^{\beta}}=g^{\alpha\bar{\beta}}\left(P'(s)e^{-s}z^{\beta}\right)=\frac{P'(s)}{u''(s)}z^{\alpha}. 
\end{align*}
Hence, $\nabla f$ is holomorphic if and only if 
\begin{align}\label{c1}
\frac{P'(s)}{u''(s)}=c
\end{align}
for some constant $c\in\mathbb{R}$. 
Substituting (\ref{f}) into (\ref{c1}), we have the following third order ODE: 
\begin{align}\label{ODE}
\frac{u'''}{u''}+\left(\frac{n-1}{u'}-c\right)u''=n-u'. 
\end{align}
Hence, we get a $U(n)$-invariant gradient shrinking K\"ahler Ricci soliton structure on $N^{n}_{k}$ when we find a solution $u$ of (\ref{ODE}) which satisfies 
condition (\ref{udash}) and (\ref{uasy}) for some $c\in\mathbb{R}$. 
Then, Cao \cite{Cao} proved the following. 
\begin{theorem}[\cite{Cao}]\label{caothm}
There exists one and only one pair $(u,c)$ so that $u$ and $c$ satisfies (\ref{udash}), (\ref{uasy}) and (\ref{ODE}). 
Additionally, it follows that $0<c<1$. 
\end{theorem}
Thus, there exists a unique $U(n)$-invariant gradient shrinking K\"ahler Ricci soliton structure on $N_{k}^{n}$, and K\"ahler form $\omega$ and potential function $f$ are written as (\ref{omega}) 
and (\ref{F}) on $(\mathbb{C}^{n}\setminus\{0\})/\mathbb{Z}_{k}$, respectively. 
%%%%%%%%%%%%%%%%%%%%%%%%%%%%%%%%%%%%%%%%%%%%%%%%%%%%%%%%%%%%%%%%%%%%%%%%%%%%%%%%%%%%%%%%%%%%%%%%%%%%%%%%
%%%%%%%%%%%%%%%%%%%%%%%%%%%%%%%%%%%%%%%%%%%%%%%%%%%%%%%%%%%%%%%%%%%%%%%%%%%%%%%%%%%%%%%%%%%%%%%%%%%%%%%%
%%%%%%%%%%%%%%%%%%%%%%%%%%%%%%%%%%%%%%%%%%%%%%%%%%%%%%%%%%%%%%%%%%%%%%%%%%%%%%%%%%%%%%%%%%%%%%%%%%%%%%%%
%%%%%%%%%%%%%%%%%%%%%%%%%%%%%%%%%%%%%%%%%%%%%%%%%%%%%%%%%%%%%%%%%%%%%%%%%%%%%%%%%%%%%%%%%%%%%%%%%%%%%%%%
%%%%%%%%%%%%%%%%%%%%%%%%%%%%%%%%%%%%%%%%%%%%%%%%%%%%%%%%%%%%%%%%%%%%%%%%%%%%%%%%%%%%%%%%%%%%%%%%%%%%%%%%
%%%%%%%%%%%%%%%%%%%%%%%%%%%%%%%%%%%%%%%%%%%%%%%%%%%%%%%%%%%%%%%%%%%%%%%%%%%%%%%%%%%%%%%%%%%%%%%%%%%%%%%%
%%%%%%%%%%%%%%%%%%%%%%%%%%%%%%%%%%%%%%%%%%%%%%%%%%%%%%%%%%%%%%%%%%%%%%%%%%%%%%%%%%%%%%%%%%%%%%%%%%%%%%%%
%%%%%%%%%%%%%%%%%%%%%%%%%%%%%%%%%%%%%%%%%%%%%%%%%%%%%%%%%%%%%%%%%%%%%%%%%%%%%%%%%%%%%%%%%%%%%%%%%%%%%%%%
%%%%%%%%%%%%%%%%%%%%%%%%%%%%%%%%%%%%%%%%%%%%%%%%%%%%%%%%%%%%%%%%%%%%%%%%%%%%%%%%%%%%%%%%%%%%%%%%%%%%%%%%
%%%%%%%%%%%%%%%%%%%%%%%%%%%%%%%%%%%%%%%%%%%%%%%%%%%%%%%%%%%%%%%%%%%%%%%%%%%%%%%%%%%%%%%%%%%%%%%%%%%%%%%%
%%%%%%%%%%%%%%%%%%%%%%%%%%%%%%%%%%%%%%%%%%%%%%%%%%%%%%%%%%%%%%%%%%%%%%%%%%%%%%%%%%%%%%%%%%%%%%%%%%%%%%%%
\section{Lens spaces in $N^{n}_{k}$ as self-similar solutions}\label{lenssp}
In this section we see that a lens space $L(k\,;1)(r)$ with radius $r$ is embedded in $(N_{k}^{n},\omega,f)$ as a self-similar solution, 
and whether it is a self-shrinker or self-expander is determined by its radius $r$. 
Actually, we prove that there exists the specific radius $r_{2}$ such that 
$L(k\,;1)(r)$ is a self-shrinker or self-expander if $r<r_{2}$ or $r>r_{2}$, respectively. 

Let $p$, $q_{1},\dots,q_{n}$ be integers such that $q_{i}$ are coprime to $p$, and $r$ be a positive constant. 
Then, the lens space $L(p\,;q_{1},\dots,q_{n})(r)$ with radius $r$ is the quotient of $S^{2n-1}(r) \subset \mathbb{C}^{n}$, the sphere with radius $r$, 
by the free $\mathbb{Z}_{p}$-action defined by 
\[[\ell]\cdot (z^{1},\dots,z^{n}):=(e^{2\pi i\ell\frac{q_{1}}{p}}z^{1},\dots,e^{2\pi i\ell\frac{q_{n}}{p}}z^{n}). \]
We restrict ourselves to the case that given integers $n$ and $k$ satisfy $n\geq 2$ and $1\leq k \leq n-1$. 
We write 
\[L(k\,;1)(r):=L(k\,;\overbrace{1,\dots,1}^{n})(r), \]
for short. 
It is clear that $L(k\,;1)(r)$ is embedded in $(\mathbb{C}^{n}\setminus\{0\})/\mathbb{Z}_{k}$, and 
$U(n)$ acts on $L(k\,;1)(r)$ transitively, since $\mathbb{Z}_{k}$-action defined by
\[[\ell]\cdot (z^{1},\dots,z^{n}):=(e^{2\pi i\ell\frac{1}{k}}z^{1},\dots,e^{2\pi i\ell\frac{1}{k}}z^{n})\]
and $U(n)$-action commute. 

Let $(N_{k}^{n},\omega)$ be the unique $U(n)$-invariant gradient shrinking K\"ahler Ricci soliton with potential function $f$ given in Theorem \ref{caothm}. 
As explained in Section \ref{caoconst}, we have an open dense subset 
\[(\mathbb{C}^{n}\setminus\{0\})/\mathbb{Z}_{k} \subset N_{k}^{n}. \]
Via this identification, we embed $L(k\,;1)(r)$ into $N_{k}^{n}$, and denote its inclusion map by 
\[\iota_{r}:L(k\,;1)(r) \hookrightarrow N_{k}^{n}. \]
Actually, $L(k\,;1)(r)$ is given as a level set of potential function $f$. 
\begin{lemma}
We have 
\[L(k\,;1)(r)=\{\,f=\gamma\,\}, \]
where $\gamma:=P(\log r^2)$ and $P$ is given by (\ref{f}). 
\end{lemma}
\begin{proof}
It is clear that $L(k\,;1)(r)$ is contained in $\{\,f=\gamma\,\}$ by a relation $f(z)=P(s)$ with $s=\log|z|^2$. 
To show the converse inclusion, it is sufficient to see that $P$ is strictly increasing. 
This is true since $P'>0$ by the equality (\ref{c1}), the positivity condition (\ref{udash}) and a fact that $0<c<1$ stated in Theorem \ref{caothm}. 
\end{proof}

Since $L(k\,;1)(r)$ is a level set of $f$, the second fundamental form $A$ and the mean curvature vector field $H$ of $\iota_{r}:L(k\,;1)(r) \hookrightarrow N_{k}^{n} $ are given by
\begin{align}\label{AH}
A(\iota_{r})=-\frac{\nabla f}{|\nabla f|^2}\mathop{\mathrm{Hess}}f \quad\mathrm{and}\quad H(\iota_{r})=-\frac{\nabla f}{|\nabla f|^2}\mathop{\mathrm{tr}^{\top}}\mathop{\mathrm{Hess}}f,
\end{align}
where $\nabla f$ and $\mathop{\mathrm{Hess}}f$ is the gradient and the Hessian of $f$ with respect to the ambient Riemannian metric $g$, 
and $\mathrm{tr}^{\top}$ is the trace restricted on $T_{p}L(k\,;1)(r)$ at each point $p$ in $L(k\,;1)(r)$. 
Since $L(k\,;1)(r)$ and the K\"ahler structure on $N_{k}^{n}$ are invariant under $U(n)$-action and it acts transitively on $L(k\,;1)(r)$, a function 
\[-\frac{1}{|\nabla f|^2}\mathop{\mathrm{tr}^{\top}}(\mathop{\mathrm{Hess}}f)\]
on $L(k\,;1)(r)$ is actually a constant, and we denote the constant by $\lambda(r)$. 
Thus, the embedding $\iota_{r}:L(k\,;1)(r) \hookrightarrow N_{k}^{n}$ is a self-similar solution with 
\[H(\iota_{r})=\lambda(r){\nabla f}^{\bot}. \]
Here we used that $\nabla f$ is normal to $L(k\,;1)(r)$, that is, ${\nabla f}^{\bot}=\nabla f$ actually. 
The reminder is to determine the sign of $\lambda(r)$. By the $U(n)$-invariance, it suffices to compute 
\[-\frac{1}{|\nabla f|^2}\mathop{\mathrm{tr}^{\top}}(\mathop{\mathrm{Hess}}f)\]
at a point $p=(r,0\dots,0)$ in $L(k\,;1)(r)$. Put $s:=\log r^2$ and 
\[v_{1}:=\frac{e^{\frac{s}{2}}}{\sqrt{2u''(s)}}\frac{\partial}{\partial y^{1}}=\sqrt{-1}\frac{e^{\frac{s}{2}}}{\sqrt{2u''(s)}}\left(\frac{\partial}{\partial z^{1}}-\frac{\partial}{\partial \bar{z}^{1}}\right). \]
Furthermore, put 
\begin{align*}
w_{\alpha}:=&\frac{e^{\frac{s}{2}}}{\sqrt{2u'(s)}}\frac{\partial}{\partial x^{\alpha}}=\frac{e^{\frac{s}{2}}}{\sqrt{2u'(s)}}\left(\frac{\partial}{\partial z^{\alpha}}+\frac{\partial}{\partial \bar{z}^{\alpha}}\right)\\
Jw_{\alpha}:=&\frac{e^{\frac{s}{2}}}{\sqrt{2u'(s)}}\frac{\partial}{\partial y^{\alpha}}=\sqrt{-1}\frac{e^{\frac{s}{2}}}{\sqrt{2u'(s)}}\left(\frac{\partial}{\partial z^{\alpha}}-\frac{\partial}{\partial \bar{z}^{\alpha}}\right), 
\end{align*}
for $\alpha=2,\dots,n$. Then, by (\ref{g}), one can check that $\{\, v_{1},w_{2},Jw_{2},\dots, w_{n},Jw_{n}, \}$ is an orthonormal basis of $T_{p}L(k\,;1)(r)$ at $p=(r,0\dots,0)$. 
Here, we have
\begin{align}\label{hesses}
\begin{aligned}
\mathop{\mathrm{Hess}}f(v_{1},v_{1})=&\frac{e^{s}}{u''(s)}\frac{\partial^2 f}{\partial z^{1}\partial \bar{z}^{1}}(p)=\frac{P''(s)}{u''(s)}\\
\mathop{\mathrm{Hess}}f(w_{\alpha},w_{\alpha})=&\mathop{\mathrm{Hess}}f(Jw_{\alpha},Jw_{\alpha})= \frac{e^{s}}{u'(s)}\frac{\partial^2 f}{\partial z^{\alpha}\partial \bar{z}^{\alpha}}(p)=\frac{P'(s)}{u'(s)}. 
\end{aligned}
\end{align}
Thus, we have 
\begin{align*}
\mathop{\mathrm{tr}^{\top}}\mathop{\mathrm{Hess}}f=&\mathop{\mathrm{Hess}}f(v_{1},v_{1})
+\sum_{\alpha=2}^{n}\mathop{\mathrm{Hess}}f(w_{\alpha},w_{\alpha})+\sum_{\alpha=2}^{n}\mathop{\mathrm{Hess}}f(Jw_{\alpha},Jw_{\alpha})\\
=&\frac{P''(s)}{u''(s)}+2(n-1)\frac{P'(s)}{u'(s)}. 
\end{align*}
By $P'=cu''$, we have 
\begin{align*}
\frac{P''}{u''}+2(n-1)\frac{P'}{u'}=&c\frac{u'''}{u''}+2c(n-1)\frac{u''}{u'}\\
=&c\left( \left(\frac{u'''}{u''}+(n-1)\frac{u''}{u'} \right)+(n-1)\frac{u''}{u'}  \right)\\
=&c\left(n-u'+cu''+(n-1)\frac{u''}{u'}\right), 
\end{align*}
where we used ODE (\ref{ODE}) in the last equality. 
Furthermore, It is clear that 
\begin{align}\label{nab}
\nabla f=ce^{\frac{s}{2}}\frac{\partial}{\partial x^{1}} \quad\mathrm{and}\quad |\nabla f|^2=2c^2u''(s).
\end{align}
at $p=(r,0,\dots,0)\in L(k\,;1)(r)$. 
Thus, we have 
\[\lambda(r)=\frac{-1}{2cu''(s)}\left( n-u'(s)+cu''(s)+(n-1)\frac{u''(s)}{u'(s)} \right), \]
where $s=\log r^2$. 

To capture the behavior of $\lambda(r)$, we need the following lemma. 
The radius $r_{2}$ in the statement (2) of the following lemma is needed to determine whether $L(k\,;1)(r)$ is a self-shrinker or a self-expander, 
and $r_{1}$ determines whether the Ricci-mean curvature flow of $L(k\,;1)(r)$ converges to $S_{0}$ or $S_{\infty}$. 
\begin{lemma}\label{2radii}
\indent
\begin{enumerate}
\item It holds that 
\begin{align*}
&\lambda(r)\to -\infty \quad \mathrm{and} \quad \lambda(r)=\mathcal{O}(r^{-2k}) \quad  \mathrm{as} \quad r\to 0, \\
&\lambda(r)\to \infty \quad \mathrm{and} \quad \lambda(r)=\mathcal{O}(r^{2k}) \quad  \mathrm{as} \quad r\to \infty. 
\end{align*}
\item There exists a unique pair of radii $r_{1}<r_{2}$ which satisfies the following. 
\begin{itemize}
\item $\lambda(r)\in (-\infty,-1)$ for $r\in(0,r_{1})$. 
\item $\lambda(r_{1})=-1$. 
\item $\lambda(r)\in(-1,0)$ for $r\in(r_{1},r_{2})$. 
\item $\lambda(r_{2})=0$. 
\item $\lambda(r)\in(0,\infty)$ for $r\in (r_{2},\infty)$. 
\end{itemize}
\end{enumerate}
\end{lemma}
\begin{proof}
By the asymptotic conditions (\ref{uasy}), we have 
\begin{align*}
&n-u'(s)+cu''(s)+(n-1)\frac{u''(s)}{u'(s)} \to k  \quad (s\to -\infty)\\
&n-u'(s)+cu''(s)+(n-1)\frac{u''(s)}{u'(s)} \to -k  \quad (s\to \infty), 
\end{align*}
and also have
\begin{align*}
&u''(s) \to 0 \quad \mathrm{and}\quad u''(s)=\mathcal{O}(e^{ks}) \quad (s\to -\infty)\\
&u''(s) \to 0 \quad \mathrm{and}\quad u''(s)=\mathcal{O}(e^{-ks}) \quad (s\to \infty). 
\end{align*}
Thus, we have proved the statement (1). 

To prove the statement (2), we will prove that  the derivative of $\lambda(r)$ at $r$ such that $\lambda(r)=-1$ or $\lambda(r)=0$ is positive. 
Then, combining the statement (1), this implies immediately that $\lambda(r)$ takes the value $-1$ and $0$ only once. 

Define $\Lambda(s)$ by 
\[\Lambda(s):=\lambda(r)=\frac{-1}{2cu''(s)}\left( n-u'(s)+cu''(s)+(n-1)\frac{u''(s)}{u'(s)} \right)\]
with $s=\log r^2$. 
Then, we have
\[\frac{d}{d r}\lambda(r)=2e^{-\frac{s}{2}}\frac{d}{d s}\Lambda(s). \]
Hence, the positivity of $d\lambda/dr$ is equivalent to the positivity of $d\Lambda/ds$. 
By a straightforward computation, we have
\begin{align}\label{eq31}
\begin{aligned}
\frac{d}{d s}\Lambda(s)=&\frac{-1}{2cu''(s)}\biggl(-u''(s)+cu'''(s)\\
&+(n-1)\frac{u'''(s)}{u'(s)}-(n-1)\frac{(u''(s))^2}{(u'(s))^2}\biggr)-\Lambda(s)\frac{u'''(s)}{u''(s)}\\
=&\frac{1}{2c}+\frac{(n-1)u''(s)}{2c(u'(s))^2}\\
&+\frac{1}{2cu'(s)}\biggl(-c\Bigl(1+2\Lambda(s)\Bigr)u'(s)-(n-1)\biggr)\frac{u'''(s)}{u''(s)}. 
\end{aligned}
\end{align}
By ODE (\ref{ODE}), we have
\begin{align}\label{eq32}
\begin{aligned}
\frac{u'''(s)}{u''(s)}=&n-u'(s)-\left(\frac{n-1}{u'(s)}-c\right)u''(s)\\
=&\left(n-u'(s)+cu''(s)+(n-1)\frac{u''(s)}{u'(s)}\right)-2(n-1)\frac{u''(s)}{u'(s)}\\
=&-2cu''(s)\Lambda(s)-2(n-1)\frac{u''(s)}{u'(s)}\\
=&2\Bigl(-c\Lambda(s) u'(s)-(n-1)\Bigr)\frac{u''(s)}{u'(s)}. 
\end{aligned}
\end{align}
Substituting (\ref{eq32}) into (\ref{eq31}), we have
\begin{align*}
\frac{d}{d s}\Lambda(s)=&\frac{1}{2c}+\frac{(n-1)u''(s)}{2c(u'(s))^2}\\
&+\frac{1}{c}\biggl(-c\Bigl(1+2\Lambda(s)\Bigr)u'(s)-(n-1)\biggr)\biggl(-c\Lambda(s) u'(s)-(n-1)\biggr)\frac{u''(s)}{(u'(s))^2}. 
\end{align*}
We remark that since $c, u''>0$, 
\[\frac{1}{2c}+\frac{(n-1)u''(s)}{2c(u'(s))^2}>0. \]
Since $\Lambda(s)\to -\infty$ as $s\to -\infty$ and $\Lambda(s)\to \infty$ as $s\to \infty$, there exists an $s\in\mathbb{R}$ such that $\Lambda(s)=0$, 
and for such $s$ we have
\begin{align*}
&\frac{1}{c}\biggl(-c\Bigl(1+2\Lambda(s)\Bigr)u'(s)-(n-1)\biggr)\biggl(-c\Lambda(s) u'(s)-(n-1)\biggr)\frac{u''(s)}{(u'(s))^2}\\
=&\frac{n-1}{c}\Bigl(cu'(s)+(n-1)\Bigr)\frac{u''(s)}{(u'(s))^2}>0, 
\end{align*}
since $c, u',u''>0$. 
Thus, we have proved that 
\[\frac{d}{ds}\Lambda(s)>\frac{1}{2c}+\frac{(n-1)u''(s)}{2c(u'(s))^2}>0\]
for $s$ such that $\Lambda(s)=0$. 
Similarly, there exists another $s\in\mathbb{R}$ such that $\Lambda(s)=-1$, 
and for such $s$ we have
\begin{align*}
&\frac{1}{c}\biggl(-c\Bigl(1+2\Lambda(s)\Bigr)u'(s)-(n-1)\biggr)\biggl(-c\Lambda(s) u'(s)-(n-1)\biggr)\frac{u''(s)}{(u'(s))^2}\\
=&\frac{1}{c}\Bigl(cu'(s)-(n-1)\Bigr)^2\frac{u''(s)}{(u'(s))^2}\geq 0
\end{align*}
since $c, u''>0$. 
Thus, we have proved that 
\[\frac{d}{ds}\Lambda(s)\geq \frac{1}{2c}+\frac{(n-1)u''(s)}{2c(u'(s))^2}>0\]
for $s$ such that $\Lambda(s)=-1$. 
Consequently, we have proved that 
\[\frac{d}{dr}\lambda(r)>0\]
for $r$ such that $\lambda(r)=0$ or $\lambda(r)=-1$. 
By this property and the statement (1), the statement (2) follows. 
\end{proof}

By Lemma \ref{2radii}, we have proved the following. 
This is the same as Theorem \ref{sumofmain1}. 

\begin{theorem}\label{mainthm}
For every $0<r<\infty$, the embedding
\[\iota_{r}:L(k\,;1)(r) \hookrightarrow N_{k}^{n}\]
is a compact self-similar solution with 
\[H(\iota_{r})=\lambda(r){\nabla f}^{\bot}, \]
and there exists the unique radius $r_{2}$ such that $\iota_{r}:L(k\,;1)(r) \hookrightarrow N_{k}^{n}$ is a non-minimal self-shrinker, minimal submanifold or 
non-minimal self-expander when $r<r_{2}$, $r=r_{2}$ or $r_{2}<r$, respectively. 
\end{theorem}

For the following sections, here we compute the norm of $A(\iota_{r})$. It is easy to see that $A(\iota_{r})$ is diagonalized by the  orthonormal basis $\{\, v_{1},w_{2},Jw_{2},\dots, w_{n},Jw_{n}, \}$. 
Hence, we have
\[|A(\iota_{r})|^{2}=|A(v_{1},v_{1})|^2+\sum_{\alpha=2}^{n}|A(w_{k},w_{k})|^2+\sum_{\alpha=2}^{n}|A(Jw_{k},Jw_{k})|^2. \]
By (\ref{AH}), (\ref{hesses}) and (\ref{nab}), with $s=\log r^2$, we have
\begin{align*}
|A(\iota_{r})|^{2}=\frac{1}{2c^{2}u''(s)}\left(\left(\frac{P''(s)}{u''(s)}\right)^2+2(n-1)\left(\frac{P'(s)}{u'(s)}\right)^2\right). 
\end{align*}
By $P'=cu''$ and ODE (\ref{ODE}), we have 
\begin{align*}
|A(\iota_{r})|^{2}=\frac{1}{2u''(s)}\left(\left(n-u'(s)-\left(\frac{n-1}{u'(s)}-c\right)u''(s)\right)^2+2(n-1)\left(\frac{u''(s)}{u'(s)}\right)^2\right). 
\end{align*}
By a similar argument of the proof of the statement (1) of Lemma \ref{2radii}, we can prove the following. 
\begin{lemma}\label{rateofA}
It holds that 
\begin{align*}
&|A(\iota_{r})|^{2}\to \infty \quad \mathrm{and} \quad |A(\iota_{r})|^{2}=\mathcal{O}(r^{-2k}) \quad \mathrm{as}\quad r\to 0, \\
&|A(\iota_{r})|^{2}\to \infty \quad \mathrm{and} \quad |A(\iota_{r})|^{2}=\mathcal{O}(r^{2k}) \quad \mathrm{as}\quad r\to \infty. 
\end{align*}
\end{lemma}
%%%%%%%%%%%%%%%%%%%%%%%%%%%%%%%%%%%%%%%%%%%%%%%%%%%%%%%%%%%%%%%%%%%%%%%%%%%%%%%%%%%%%%%%%%%%%%%%%%%%%%%%
%%%%%%%%%%%%%%%%%%%%%%%%%%%%%%%%%%%%%%%%%%%%%%%%%%%%%%%%%%%%%%%%%%%%%%%%%%%%%%%%%%%%%%%%%%%%%%%%%%%%%%%%
%%%%%%%%%%%%%%%%%%%%%%%%%%%%%%%%%%%%%%%%%%%%%%%%%%%%%%%%%%%%%%%%%%%%%%%%%%%%%%%%%%%%%%%%%%%%%%%%%%%%%%%%
%%%%%%%%%%%%%%%%%%%%%%%%%%%%%%%%%%%%%%%%%%%%%%%%%%%%%%%%%%%%%%%%%%%%%%%%%%%%%%%%%%%%%%%%%%%%%%%%%%%%%%%%
%%%%%%%%%%%%%%%%%%%%%%%%%%%%%%%%%%%%%%%%%%%%%%%%%%%%%%%%%%%%%%%%%%%%%%%%%%%%%%%%%%%%%%%%%%%%%%%%%%%%%%%%
%%%%%%%%%%%%%%%%%%%%%%%%%%%%%%%%%%%%%%%%%%%%%%%%%%%%%%%%%%%%%%%%%%%%%%%%%%%%%%%%%%%%%%%%%%%%%%%%%%%%%%%%
%%%%%%%%%%%%%%%%%%%%%%%%%%%%%%%%%%%%%%%%%%%%%%%%%%%%%%%%%%%%%%%%%%%%%%%%%%%%%%%%%%%%%%%%%%%%%%%%%%%%%%%%
%%%%%%%%%%%%%%%%%%%%%%%%%%%%%%%%%%%%%%%%%%%%%%%%%%%%%%%%%%%%%%%%%%%%%%%%%%%%%%%%%%%%%%%%%%%%%%%%%%%%%%%%
%%%%%%%%%%%%%%%%%%%%%%%%%%%%%%%%%%%%%%%%%%%%%%%%%%%%%%%%%%%%%%%%%%%%%%%%%%%%%%%%%%%%%%%%%%%%%%%%%%%%%%%%
%%%%%%%%%%%%%%%%%%%%%%%%%%%%%%%%%%%%%%%%%%%%%%%%%%%%%%%%%%%%%%%%%%%%%%%%%%%%%%%%%%%%%%%%%%%%%%%%%%%%%%%%
%%%%%%%%%%%%%%%%%%%%%%%%%%%%%%%%%%%%%%%%%%%%%%%%%%%%%%%%%%%%%%%%%%%%%%%%%%%%%%%%%%%%%%%%%%%%%%%%%%%%%%%%
\section{The motion by Ricci-mean curvature flow of a lens space in $N^{n}_{k}$}\label{motion}
In this section we observe how a lens space $L(k\,;1)(r)$ in $N^{n}_{k}$ moves by Ricci-mean curvature flow and what it converges to. 
Continuing Section \ref{lenssp}, let $(N_{k}^{n},\omega)$ be a unique $U(n)$-invariant gradient shrinking K\"ahler Ricci soliton with potential function $f$ given in Theorem \ref{caothm}. 
Then we have 
\[\nabla f=cr\frac{\partial}{\partial r}, \]
where $r=|z|$. Fix $T\in (0,\infty)$. One can easily see that 
\[\Phi_{t}:(\mathbb{C}^{n}\setminus\{0\})/\mathbb{Z}_{k}\to (\mathbb{C}^{n}\setminus\{0\})/\mathbb{Z}_{k}\]
defined by 
\[\Phi_{t}(z):=\kappa(t)z\quad\mathrm{with}\quad\kappa(t):=\left(\frac{T}{T-t}\right)^{\frac{{c}}{2}}\]
for $t\in [0,T)$ is the 1-parameter family of automorphisms of $N_{k}^{n}$ generated by $\frac{1}{2(T-t)}\nabla f$ with $\Phi_{0}=\mathrm{id}$. 
Then, it follows that $g_{t}:=2(T-t)\Phi_{t}^{*}g$ satisfies the Ricci flow equation:
\[\frac{\partial}{\partial t}g_{t}=-2\mathrm{Ric}(g_{t}). \]
Fix $r\in(0,\infty)$ and let $\iota_{r}:L(k\,;1)(r) \hookrightarrow N_{k}^{n}$ be a lens space with radius $r$ and 
\[F_{t}:L(k\,;1)(r)\to N_{k}^{n}\quad (t\in[0,T'))\]
be the solution of Ricci-mean curvature flow along $g_{t}=2(T-t)\Phi_{t}^{*}g$ with initial condition $F_{0}=\iota_{r}$. 
We assume that $T'(=T'(r))$ is the maximal time of existence of the solution. 
By rotationally symmetry of $L(k\,;1)(r)\subset N_{k}^{n}$, the solution $F_{t}$ is written as 
\[F_{t}(p):=h(t)p, \]
by some positive smooth function $h:[0,T')\to \mathbb{R}$. 
Then the Ricci-mean curvature flow equation for $F_{t}$ is reduced to an ODE for $h(t)$. 
\begin{proposition}
The 1-parameter family of immersions $F_{t}$ is the solution of the Ricci-mean curvature flow coupled with $g_{t}$ with initial condition $F_{0}=\iota_{r}$ if and only if 
the positive smooth function $h:[0,T')\to \mathbb{R}$ satisfies the following ODE with initial condition:
\begin{align}\label{ODEradi}
\begin{aligned}
 &h(0)=1\\
&\frac{h'(t)}{h(t)}=\frac{c}{2(T-t)}\lambda(\kappa(t)h(t)r), 
\end{aligned}
\end{align}
where $\lambda$ and $\kappa$ are given functions. 
\end{proposition}
\begin{proof}
Recall that the Ricci-mean curvature flow equation is 
\[\frac{\partial}{\partial t}F_{t}=H_{g_{t}}(F_{t}), \]
where $H_{g_{t}}(F_{t})$ is the mean curvature vector field of $F_{t}$ computed with the Riemannian metric $g_{t}=2(T-t)\Phi_{t}^{*}g$. 
It is easy to see that 
\begin{align*}
H_{g_{t}}(F_{t})=\frac{1}{2(T-t)}H(\iota_{\kappa(t)h(t)r})=\frac{c}{2(T-t)}\lambda(\kappa(t)h(t)r)\left(r\frac{\partial}{\partial r}\right), 
\end{align*}
where $H$ without subscript $g_{t}$ denotes the mean curvature vector field with respect to the original ambient metric $g$. 
Since 
\[\frac{\partial}{\partial t}F_{t}=\frac{h'(t)}{h(t)}\left(r\frac{\partial}{\partial r}\right), \]
the proposition is proved. 
\end{proof}

Put 
\begin{align}\label{Rh}
R(t):=\kappa(t)h(t)r=\left(\frac{T}{T-t}\right)^{\frac{{c}}{2}}h(t)r. 
\end{align}
Then, we have 
\[\frac{R'(t)}{R(t)}=\frac{c}{2(T-t)}+\frac{h'(t)}{h(t)}. \]
Thus, we have the following. 

\begin{lemma}
The ODE (\ref{ODEradi}) for $h(t)$ with initial condition is equivalent to the following ODE for $R(t)$ with initial condition: 
\begin{align}\label{ODEradi2}
\begin{aligned}
 &R(0)=r\\
&\frac{R'(t)}{R(t)}=\frac{c}{2(T-t)}\Bigl(\lambda(R(t))+1\Bigr). 
\end{aligned}
\end{align}
\end{lemma}

Therefore, analysis of the motion of Ricci-mean curvature flow emanating from $L(k\, ;1)(r)$ is reduced to the analysis of $R(t)$.
Let $r_{1}$ be the specific radius introduced in Lemma \ref{2radii}. Then, we have the following. 

\begin{lemma}\label{mainlemma}
\indent
\begin{enumerate}
\item If $r<r_{1}$, then $T'<T$ and the solution $R(t)$ of (\ref{ODEradi2}) satisfies $R(t)\to 0$ and $h(t)\to 0$ as $t\to T'$. 
Furthermore, we have 
\[(R(t))^{-2k}=\mathcal{O}\left(\frac{1}{T'-t}\right) \quad \mathrm{as} \quad t\to T'. \]
\item If $r=r_{1}$, then $T'=T$ and $R(t)=r_{1}$ is the stationary solution of (\ref{ODEradi2}) and $h(t)=\kappa^{-1}(t)\to 0$ as $t \to T$. 
\item If $r_{1}<r$, then $T'<T$ and the solution $R(t)$ of (\ref{ODEradi2}) satisfies $R(t)\to \infty$ and $h(t)\to \infty$ as $t\to T'$. 
Furthermore, we have 
\[(R(t))^{2k}=\mathcal{O}\left(\frac{1}{T'-t}\right) \quad \mathrm{as} \quad t\to T'. \]
\end{enumerate}
\end{lemma}
\begin{proof}
The proof is done by an ordinary argument for the bifurcation phenomenon of an ODE. 
First, we prove the statement (1). 
Assume that $r<r_{1}$. 
In this case, by Lemma \ref{2radii}, there exists a constant $\alpha=\alpha(r)<-1$ such that $\lambda(r)\leq \alpha$ for all $r\in (0,r_{0}]$. 
At $t=0$, we have $R'(0)<0$ by OED (\ref{ODEradi2}). 
If there exists some $t_{0}\in (0,T')$ such that $R(t_{0})=r$, it follows that 
\[R'(t_{0})=\frac{c}{2(T-t_{0})}\Bigl(\lambda(r)+1\Bigr)r< 0. \]
This means that $R(t)\in(0,r]$ for all $t\in [0,T')$ and $R(t)$ is monotonically decreasing. 
By ODE (\ref{ODEradi2}), we have
\begin{align}\label{int0}
\frac{R'(t)}{R(t)\Bigl(\lambda(R(t))+1\Bigr)}=\frac{c}{2(T-t)}, 
\end{align}
and integrating both sides from $t=0$ to $t=T'-0$ we have
\begin{align}\label{intint}
\int_{r}^{R(T'-0)}\frac{1}{R\left(\lambda(R)+1\right)}dR=\int_{0}^{T'-0}\frac{c}{2(T-t)}dt. 
\end{align}
By (1) of Lemma \ref{2radii}, we have 
\begin{align}\label{order0}
\frac{1}{R\left(\lambda(R)+1\right)}=\mathcal{O}(R^{2k-1})\quad\mathrm{as}\quad R\to 0. 
\end{align}
Thus, the left hand side of (\ref{intint}) is integrable, and we have proved that $T'<T$. 
If $\lim_{t\to T'}R(t)>0$ then $R'(t)$ is bounded as $t\to T'$ by ODE (\ref{ODEradi2}), 
and this contradicts that $T'$ is the maximal time of existence of the solution. 
Thus, it holds that $R(t)\to 0$ as $t\to T'$. 
Integrating both sides of (\ref{int0}) from $t$ to $T'$ and combining the estimate of order (\ref{order0}) and $R(t)\to 0$ as $t\to T'$, we have 
\[C(R(t))^{2k}\geq \int_{t}^{T'}\frac{c}{2(T-t)}dt\geq \frac{c}{2T}(T'-t)\]
with some constant $C>0$ and $t$ sufficiently close to $T'$. Thus, we have 
\[(R(t))^{-2k}=\mathcal{O}\left(\frac{1}{T'-t}\right) \quad \mathrm{as} \quad t\to T'. \]
Since $R(t)=\kappa(t)h(t)r$, $R(t)\to 0$ as $t\to T'$ and $\kappa(t)$ is bounded on $[0,T')$, it holds that $h(t)\to 0$ as $t\to T'$. 
Hence, we have proved the statement (1). 

The statement (2) is clear since $\lambda(r_{1})+1=0$. 

Finally, we prove the statement (3). The argument is very similar to the proof of the statement (1). Assume that $r_{1}<r$. 
In this case, by Lemma \ref{2radii}, there exists a constant $\alpha=\alpha(r_{0})>-1$ such that $\lambda(r)\geq \alpha$ for all $r\in [r_{0},\infty)$. 
At $t=0$, we have $R'(0)>0$ by OED (\ref{ODEradi2}). 
If there exists some $t_{0}\in (0,T')$ such that $R(t_{0})=r$, it follows that 
\[R'(t_{0})=\frac{c}{2(T-t_{0})}\Bigl(\lambda(r)+1\Bigr)r>0. \]
This means that $R(t)\in[r,\infty)$ for all $t\in [0,T')$ and $R(t)$ is monotonically increasing. 
By (1) of Lemma \ref{2radii}, we have 
\begin{align}\label{orderinfty}
\frac{1}{R\left(\lambda(R)+1\right)}=\mathcal{O}(R^{-2k-1})\quad\mathrm{as}\quad R\to \infty. 
\end{align}
Thus, the left hand side of (\ref{intint}) is integrable, and we have proved that $T'<T$. 
If $\lim_{t\to T'}R(t)<\infty$ then $R'(t)$ is bounded as $t\to T'$ by ODE (\ref{ODEradi2}), 
and this contradicts that $T'$ is the maximal time of existence of the solution. 
Thus, it holds that $R(t)\to \infty$ as $t\to T'$. 
Integrating both sides of (\ref{int0}) from $t$ to $T'$ and combining the estimate of order (\ref{orderinfty}) and $R(t)\to \infty$ as $t\to T'$, we have 
\[C(R(t))^{-2k}\geq \int_{t}^{T'}\frac{c}{2(T-t)}dt=\frac{c}{2T}(T'-t)\]
with some constant $C>0$ and $t$ sufficient close to $T'$. Thus, we have 
\[(R(t))^{2k}=\mathcal{O}\left(\frac{1}{T'-t}\right) \quad \mathrm{as} \quad t\to T'. \]
Since $R(t)=\kappa(t)h(t)r$, $R(t)\to \infty$ as $t\to T'$ and $\kappa(t)$ is bounded on $[0,T')$, it holds that $h(t)\to \infty$ as $t\to T'$. 
Hence, we have proved the statement (3). 
\end{proof}

\begin{remark}\label{T}
In the case $r<r_{1}$, by integrating both sides of (\ref{int0}) from $t=0$ to $t=T'$ and straightforward computation, 
the maximal time $T'(=T'(r))$ is explicitly given as 
\[T'=T-T\exp\left( \frac{2}{c}\int_{0}^{r}\frac{1}{R\left(\lambda(R)+1\right)}dR  \right). \]
Similarly, in the case $r_{1}<r$, the maximal time $T'(=T'(r))$ is explicitly given as
\[T'=T-T\exp\left( \frac{-2}{c}\int_{r}^{\infty}\frac{1}{R\left(\lambda(R)+1\right)}dR  \right). \]
\end{remark}

By Lemma \ref{mainlemma}, we can prove the following theorem. This contains Theorem \ref{sumofmain2}. 

\begin{theorem}\label{mainRM}
\indent
\begin{enumerate}
\item If $r<r_{1}$, then $T'<T$ and $F_{t}:L(k\,;1)(r)\to N_{k}^{n}$ converges pointwise to $S_{0}$ as $t\to T'$. 
Furthermore, we have $|A_{g_{t}}(F_{t})|_{g_{t}}^2\to \infty$ as $t\to T'$ and 
there exists some constant $C>0$ such that 
\[|A_{g_{t}}(F_{t})|_{g_{t}}^2\leq \frac{C}{T'-t}\quad \mathrm{on} \quad [0,T'). \] 
\item If $r=r_{1}$, then $T'=T$ and $F_{t}:L(k\,;1)(r)\to N_{k}^{n}$ is given by $F_{t}(p)=\kappa^{-1}(t)p$ and converges pointwise to $S_{0}$ as $t\to T$. 
Furthermore, we have $|A_{g_{t}}(F_{t})|_{g_{t}}^2\to \infty$ as $t\to T$ and 
there exists some constant $C>0$ such that 
\[|A_{g_{t}}(F_{t})|_{g_{t}}^2=\frac{C}{T'-t}\quad \mathrm{on} \quad [0,T'). \] 
\item If $r_{1}<r$, then $T'<T$ and $F_{t}:L(k\,;1)(r)\to N_{k}^{n}$ converges pointwise to $S_{\infty}$ as $t\to T'$. 
Furthermore we have $|A_{g_{t}}(F_{t})|_{g_{t}}^2\to \infty$ as $t\to T'$ and 
there exists some constant $C>0$ such that 
\[|A_{g_{t}}(F_{t})|_{g_{t}}^2\leq\frac{C}{T'-t}\quad \mathrm{on} \quad [0,T'). \] 
\end{enumerate}
\end{theorem}
\begin{proof}
It is easy to see that 
\begin{align}\label{AR}
|A_{g_{t}}(F_{t})|_{g_{t}}^2=\frac{1}{2(T-t)}|A(\iota_{\kappa(t)h(t)r})|^2=\frac{1}{2(T-t)}|A(\iota_{R(t)})|^2, 
\end{align}
where $|A_{g_{t}}|_{g_{t}}$ and $|A|$ is the norm of the second fundamental form with respect to the ambient metric $g_{t}$ and $g$, respectively. 
In the case (1), we have $h(t)\to 0$ as $t\to T'$ by Lemma \ref{mainlemma}. 
Let $p=(z^{1},\dots,z^{n})$ be a point in $ L(k\,;1)(r) \subset (\mathbb{C}^{n}\setminus\{0\})/\mathbb{Z}_{k}$ and assume that $z^{j}\neq 0$ for some $j$. 
Then, $F_{t}(p)=h(t)p$ is identified with
\[((h(t)z^{1}:\dots:h(t)z^{n}),(1:(h(t)z^{j})^{k}))=((z^{1}:\dots:z^{n}),(1:(h(t)z^{j})^{k}))\]
in $U_{j}\times\mathbb{P}^{1}$ via $\psi$. 
Hence, it is clear that $F_{t}:L(k\,;1)(r)\to N_{k}^{n}$ converges pointwise to $S_{0}$, the 0-section, as $t\to T'$. 
By the formula (\ref{AR}) and Lemma \ref{rateofA} and Lemma \ref{mainlemma}, the remaining of the statement (1) is clear. 
The proof of (3) is similar. 
In the case (2), it is enough to put $C:=2|A(\iota_{r_{1}})|^2$. 
\end{proof}

\begin{remark}
Consider the standard Hopf fibration $S^{2n-1}(r)\to \mathbb{P}^{n-1}$. 
When $r\to 0$, the total space $S^{2n-1}(r)$ collapses to $\mathbb{P}^{n-1}$ and 
this collapsing is just caused by degeneration of $S^{1}$-fiber. 
From this view point, the picture of the collapsing of $L(k\,;1)(r)$ to $S_{0}$ or $S_{\infty}$ (these are diffeomorphic to $\mathbb{P}^{n-1}$) 
is considered as a $\mathbb{Z}_{k}$-quotient analog of the collapsing of the Hopf fibration. 
\end{remark}
%%%%%%%%%%%%%%%%%%%%%%%%%%%%%%%%%%%%%%%%%%%%%%%%%%%%%%%%%%%%%%%%%%%%%%%%%%%%%%%%%%%%%%%%%%%%%%%%%%%%%%%%
%%%%%%%%%%%%%%%%%%%%%%%%%%%%%%%%%%%%%%%%%%%%%%%%%%%%%%%%%%%%%%%%%%%%%%%%%%%%%%%%%%%%%%%%%%%%%%%%%%%%%%%%
%%%%%%%%%%%%%%%%%%%%%%%%%%%%%%%%%%%%%%%%%%%%%%%%%%%%%%%%%%%%%%%%%%%%%%%%%%%%%%%%%%%%%%%%%%%%%%%%%%%%%%%%
%%%%%%%%%%%%%%%%%%%%%%%%%%%%%%%%%%%%%%%%%%%%%%%%%%%%%%%%%%%%%%%%%%%%%%%%%%%%%%%%%%%%%%%%%%%%%%%%%%%%%%%%
%%%%%%%%%%%%%%%%%%%%%%%%%%%%%%%%%%%%%%%%%%%%%%%%%%%%%%%%%%%%%%%%%%%%%%%%%%%%%%%%%%%%%%%%%%%%%%%%%%%%%%%%
%%%%%%%%%%%%%%%%%%%%%%%%%%%%%%%%%%%%%%%%%%%%%%%%%%%%%%%%%%%%%%%%%%%%%%%%%%%%%%%%%%%%%%%%%%%%%%%%%%%%%%%%
%%%%%%%%%%%%%%%%%%%%%%%%%%%%%%%%%%%%%%%%%%%%%%%%%%%%%%%%%%%%%%%%%%%%%%%%%%%%%%%%%%%%%%%%%%%%%%%%%%%%%%%%
%%%%%%%%%%%%%%%%%%%%%%%%%%%%%%%%%%%%%%%%%%%%%%%%%%%%%%%%%%%%%%%%%%%%%%%%%%%%%%%%%%%%%%%%%%%%%%%%%%%%%%%%
%%%%%%%%%%%%%%%%%%%%%%%%%%%%%%%%%%%%%%%%%%%%%%%%%%%%%%%%%%%%%%%%%%%%%%%%%%%%%%%%%%%%%%%%%%%%%%%%%%%%%%%%
%%%%%%%%%%%%%%%%%%%%%%%%%%%%%%%%%%%%%%%%%%%%%%%%%%%%%%%%%%%%%%%%%%%%%%%%%%%%%%%%%%%%%%%%%%%%%%%%%%%%%%%%
%%%%%%%%%%%%%%%%%%%%%%%%%%%%%%%%%%%%%%%%%%%%%%%%%%%%%%%%%%%%%%%%%%%%%%%%%%%%%%%%%%%%%%%%%%%%%%%%%%%%%%%%
\section{The motion by mean curvature flow of a lens space in $N^{n}_{k}$}\label{motionh}
In this section, we observe how a lens space in $N^{n}_{k}$ moves by mean curvature flow. 
As in Section \ref{motion}, let $(N_{k}^{n},\omega)$ be a unique $U(n)$-invariant gradient shrinking K\"ahler Ricci soliton with potential function $f$ given in Theorem \ref{caothm}, 
and for a given $r>0$ let $\iota_{r}:L(k\,;1)(r) \hookrightarrow N_{k}^{n}$ be a lens space with radius $r$.  
Then, by rotationally symmetry, the solution 
\[F_{t}:L(k\,;1)(r) \to N_{k}^{n} \]
of the mean curvature flow equation 
\[\frac{\partial}{\partial t}F_{t}=H(F_{t})\]
is given by 
\[F_{t}(p)=h(t)p\]
with some positive smooth function $h:[0,T')\to\mathbb{R}$ which satisfies the following ODE with initial condition:
\begin{align}\label{ODEradih}
\begin{aligned}
 &h(0)=1\\
&\frac{h'(t)}{h(t)}=c\lambda(h(t)r). 
\end{aligned}
\end{align}
We remark that this is an autonomous differential equation. 
We assume that $T'(=T'(r))$ is the maximal time of existence of the solution. 
Let $r_{2}$ be the specific radius introduced in Lemma \ref{2radii}. Then, by similar arguments as the proof of Lemma \ref{mainlemma}, one can prove the following. 
\begin{lemma}\label{mainlemma2}
\indent
\begin{enumerate}
\item If $r<r_{2}$, then $T'<\infty$ and the solution $h(t)$ of (\ref{ODEradih}) satisfies 
\[h(t)\to 0 \quad \mathrm{and} \quad (h(t))^{-2k}=\mathcal{O}\left(\frac{1}{T'-t}\right) \quad \mathrm{as} \quad t\to T'. \]
\item If $r=r_{2}$, then $T'=\infty$ and $h(t)=1$ is the stationary solution of (\ref{ODEradih}). 
\item If $r_{2}<r$, then $T'<\infty$ and the solution $h(t)$ of (\ref{ODEradih}) satisfies 
\[h(t)\to \infty \quad \mathrm{and} \quad (h(t))^{2k}=\mathcal{O}\left(\frac{1}{T'-t}\right) \quad \mathrm{as} \quad t\to T'. \]
\end{enumerate}
\end{lemma}
Furthermore, as the proof of Theorem \ref{mainRM}, combining Lemma \ref{rateofA} and Lemma \ref{mainlemma2}, we can prove the following. 
\begin{theorem}\label{mainM}
\indent
\begin{enumerate}
\item If $r<r_{2}$, then $T'<\infty$ and $F_{t}:L(k\,;1)(r)\to N_{k}^{n}$ converges pointwise to $S_{0}$ as $t\to T'$. 
Furthermore, we have $|A(F_{t})|^2\to \infty$ as $t\to T'$ and 
there exists some constant $C>0$ such that 
\[|A(F_{t})|^2\leq \frac{C}{T'-t}\quad \mathrm{on} \quad [0,T'), \]
that is, $F_{t}$ develops singularities of type I. 
\item If $r=r_{2}$, then $F_{t}\equiv F_{0}:L(k\,;1)(r)\to N_{k}^{n}$ $(t\in[0,\infty))$ is the stationary solution of the mean curvature flow since $F_{0}$ is a minimal immersion. 
\item If $r_{2}<r$, then $T'<\infty$ and $F_{t}:L(k\,;1)(r)\to N_{k}^{n}$ converges pointwise to $S_{\infty}$ as $t\to T'$. 
Furthermore we have $|A(F_{t})|^2\to \infty$ as $t\to T'$ and 
there exists some constant $C>0$ such that 
\[|A(F_{t})|^2\leq\frac{C}{T'-t}\quad \mathrm{on} \quad [0,T'), \]
that is, $F_{t}$ develops singularities of type I. 
\end{enumerate}
\end{theorem}

\begin{remark}
As Remark \ref{T}, the maximal time $T'(=T'(r))$ is explicitly given as follows. 
When $r<r_{2}$, 
\[T'=\frac{-1}{c}\int_{0}^{r}\frac{1}{r\lambda(r)}dr, \]
and when $r_{2}<r$, 
\[T'=\frac{1}{c}\int_{r}^{\infty}\frac{1}{r\lambda(r)}dr. \]
\end{remark}
%%%%%%%%%%%%%%%%%%%%%%%%%%%%%%%%%%%%%%%%%%%%%%%%%%%%%%%%%%%%%%%%%%%%%%%%%%%%%%%%%%%%%%%%%%%%%%%%%%%%%%%%
%%%%%%%%%%%%%%%%%%%%%%%%%%%%%%%%%%%%%%%%%%%%%%%%%%%%%%%%%%%%%%%%%%%%%%%%%%%%%%%%%%%%%%%%%%%%%%%%%%%%%%%%
%%%%%%%%%%%%%%%%%%%%%%%%%%%%%%%%%%%%%%%%%%%%%%%%%%%%%%%%%%%%%%%%%%%%%%%%%%%%%%%%%%%%%%%%%%%%%%%%%%%%%%%%
%%%%%%%%%%%%%%%%%%%%%%%%%%%%%%%%%%%%%%%%%%%%%%%%%%%%%%%%%%%%%%%%%%%%%%%%%%%%%%%%%%%%%%%%%%%%%%%%%%%%%%%%
%%%%%%%%%%%%%%%%%%%%%%%%%%%%%%%%%%%%%%%%%%%%%%%%%%%%%%%%%%%%%%%%%%%%%%%%%%%%%%%%%%%%%%%%%%%%%%%%%%%%%%%%
%%%%%%%%%%%%%%%%%%%%%%%%%%%%%%%%%%%%%%%%%%%%%%%%%%%%%%%%%%%%%%%%%%%%%%%%%%%%%%%%%%%%%%%%%%%%%%%%%%%%%%%%
%%%%%%%%%%%%%%%%%%%%%%%%%%%%%%%%%%%%%%%%%%%%%%%%%%%%%%%%%%%%%%%%%%%%%%%%%%%%%%%%%%%%%%%%%%%%%%%%%%%%%%%%
%%%%%%%%%%%%%%%%%%%%%%%%%%%%%%%%%%%%%%%%%%%%%%%%%%%%%%%%%%%%%%%%%%%%%%%%%%%%%%%%%%%%%%%%%%%%%%%%%%%%%%%%
%%%%%%%%%%%%%%%%%%%%%%%%%%%%%%%%%%%%%%%%%%%%%%%%%%%%%%%%%%%%%%%%%%%%%%%%%%%%%%%%%%%%%%%%%%%%%%%%%%%%%%%%
%%%%%%%%%%%%%%%%%%%%%%%%%%%%%%%%%%%%%%%%%%%%%%%%%%%%%%%%%%%%%%%%%%%%%%%%%%%%%%%%%%%%%%%%%%%%%%%%%%%%%%%%
%%%%%%%%%%%%%%%%%%%%%%%%%%%%%%%%%%%%%%%%%%%%%%%%%%%%%%%%%%%%%%%%%%%%%%%%%%%%%%%%%%%%%%%%%%%%%%%%%%%%%%%%

%%%%%%%%%%%%%%%%%%%%%%%%%%%%%%%%%%%%%%%%%%%%%%%%%%%%%%%%%%%%
%%%%%%%%%%%%%%%%%%%%%%%%%%%%%%%%%%%%%%%%%%%%%%%%%%%%%%%%%%%%
%%%%%%%%%%%%%%%%%%%%%%%%%%%%%%%%%%%%%%%%%%%%%%%%%%%%%%%%%%%%
%%%%%%%%%%%%%%%%%%%%%%%%%%%%%%%%%%%%%%%%%%%%%%%%%%%%%%%%%%%%
%%%%%%%%%%%%%%%%%%%%%%%%%%%%%%%%%%%%%%%%%%%%%%%%%%%%%%%%%%%%
\end{document}